\documentclass[10pt,letterpaper]{amsart} 

\usepackage{color}
\usepackage{multicol}
\usepackage[letterpaper, margin=0.83in]{geometry}
\usepackage{subfig}
\usepackage[overload]{empheq}

\usepackage{listings}
\usepackage{amsmath, amsthm, amssymb, wasysym, verbatim, bbm, color, graphics}
\usepackage{url}
\usepackage{mathtools}
\usepackage[colorlinks,linkcolor=blue]{hyperref}
\usepackage{romannum}
\usepackage{xfrac}
\usepackage{float}
\usepackage{amsmath,bm}
\usepackage[title]{appendix}
\usepackage{todonotes}

\DeclarePairedDelimiter{\floor}{\lfloor}{\rfloor}

\DeclareMathOperator{\tr}{\text{tr}}

\newcommand{\dom}{\Omega}

\newcommand{\Qnew}{Q^{n+1}}
\newcommand{\Qold}{Q^n}
\newcommand{\Pold}{P(Q^n)}
\newcommand{\rnew}{r^{n+1}}
\newcommand{\rold}{r^n}
\newcommand{\Qsol}{Q_{\Delta t}}
\newcommand{\rsol}{r_{\Delta t}}

\newcommand{\Qsolsub}{Q_{\Delta t_m}}
\newcommand{\rsolsub}{r_{\Delta t_m}}

\newtheorem{lemma}{Lemma}[section]

\newtheorem{corollary}[lemma]{Corollary}
\newtheorem{theorem}[lemma]{Theorem}

\newtheorem*{maintheorem*}{Main Theorem}
\theoremstyle{definition}{}

\theoremstyle{note}{
\newtheorem*{claim*}{Claim}}

\allowdisplaybreaks

\numberwithin{equation}{section}

\title[Analysis for Q-tensor flow]{On convergence of an unconditional stable numerical scheme for Q-tensor flow based on invariant quardratization method}
\date{\today}

\author[Y. Yue]{Yukun Yue}
\address[Yukun Yue]{\newline Department of Mathematical Sciences \newline Carnegie Mellon University \newline 5000 Forbes Avenue, Pittsburgh, PA 15213, USA.}
\email[]{yukuny@andrew.cmu.edu}

\begin{document}
 \pagenumbering{arabic}
\maketitle
\begin{abstract}
We present convergence analysis towards a numerical scheme designed for Q-tensor flows of nematic liquid crystals. This scheme is based on the Invariant Energy Quadratization method, which introduces an auxiliary variable to replace the original energy functional. In this work, we have shown that given an initial value with $H^2$ regularity, we can obtain a uniform $H^2$ estimate on the numerical solutions for Q-tensor flows and then deduce the convergence to a strong solution of the parabolic-type Q-tensor equation. We have also shown that the limit of the auxiliary variable is equivalent to the original energy functional term in the strong sense.
\end{abstract}

\section{Introduction}
Liquid crystal is an intermediate phase of matter between solid and liquid, which can be depicted as calamitic (rod-like) molecules in a microscopic sense\cite{collings1997introduction}. This paper will focus on the Q-tensor model for nematic liquid crystals, established by Landau-de Gennes's theory\cite{de1993physics, Intro_Qtensor}. In this framework, the director field of liquid crystal, denoted by $n$, is determined through a symmetric, trace-free $d\times d$ matrix $Q$ known as a Q-tensor order parameter\cite{Majumdar2010LandauDeGennes}. This matrix is assumed to minimize the Landau-de Gennes free energy
\begin{equation}
    \label{eq:LG_energy}
    E_{LG}(Q)=\frac{L}{2}\|\nabla Q\|_{L^2}^2+\int_\Omega \mathcal{F}_B(Q).
\end{equation}
Here $L>0$, $\Omega\in \mathbb{R}^d$ with $d=2$ or $3$ represents the spatial region where the liquid crystal molecules immerse and $\partial\dom\in C^2$. $\mathcal{F}_B$ denotes a bulk potential given by a truncated Taylor series of the thermotropic energy at $Q=0$\cite{PhysRevLett.59.2582}, given by $
\mathcal{F}_B(Q)=\frac{a}{2}\tr(Q^2)-\frac{b}{3}\tr(Q^3)+\frac{c}{4}\left(\tr(Q^2)\right)^2$,
where $a, b,$ and $c$ are constants and $c>0$. Modeled by calculating the gradient flow\cite{Ball2017, Intro_Qtensor} of Landau-de Gennes free energy, the following equation will describe the non-equilibrium situation satisfied by the Q-tensor, 
\begin{equation}
    \label{eq:Q_formula}
    Q_t=-\frac{\delta E_{LG}}{\delta Q}=M\left[L\Delta Q-\left(aQ-b\left(Q^2-\frac{1}{3}\tr(Q^2)I\right)+c\tr(Q^2)Q \right)\right]\coloneqq M\left(L\Delta Q-S(Q)\right).
\end{equation}
where $M>0$ is a constant. We prescribe $Q$ with Dirichlet or Neumann boundary condition and denote the initial value as $Q_{0}\in H^2$.

Abundance analysis results have been established for this Q-tensor flow model and the related hydrodynamics models. See \cite{AbelsLiu2013, AbelsLiu2014, CavaterraWu, Gonzalez, Iyer2015, PaicuZarnescuGlobalExistence, WangWangZhang2017} and the reference within. Various numerical approaches have also been made\cite{LiuWalkington2000,Mori_1999, Nochetto2017, ShenJie2019, Shin2008,ZHAO2017803,ZhaoYangLiWang2016}. A typical problem in designing a stable and efficient numerical scheme for problem \eqref{eq:Q_formula} is the high non-linearity of the functional derivative of the bulk potential term. The Invariant Energy Quadratization(IEQ) method, which is recently developed, is a powerful tool for dealing with such difficulty and constructing linear energy-stable schemes. It has been widely used in treating gradient flow type problems. See \cite{LEE2022108161, LIU2019206,YANG2016294, ZhaoYang2017,YangZhaoWangShen2017, ZHAO2021107331} for more applications. 

This method introduces an auxiliary variable replacing the original bulk potential. Specifically, we define
\begin{equation}
    \label{eq:r}
    r(Q)=\sqrt{2\left( \frac{a}{2}\tr(Q^2)-\frac{b}{3}\tr(Q^3)+\frac{c}{4}\left(\tr(Q^2)\right)^2 \right)+A_0},
\end{equation}
where $A_0>0$ is a large enough constant to ensure $r(Q)$ to be positive for any symmetric, trace-free tensor $Q$. This is well-defined since the bulk potential term $\mathcal{F}_B $ is bounded from below when $c>0$\cite[Theorem 2.1]{ZHAO2017803}. It follows that
\begin{equation}
    \label{eq:P}
    P(Q)\coloneqq \frac{\delta r(Q)}{\delta Q}=\frac{S(Q)}{r(Q)}=\frac{aQ-b\left(Q^2-\frac{1}{3}\tr(Q^2)I\right)+c\tr(Q^2)Q}{\sqrt{2\left( \frac{a}{2}\tr(Q^2)-\frac{b}{3}\tr(Q^3)+\frac{c}{4}\left(\tr(Q^2)\right)^2 \right)+A_0}},
\end{equation}
and $P(Q)$ is symmetric, trace-free. Then we can reformulate the equation \eqref{eq:Q_formula} as a system for $(Q,r)$ satisfying
\begin{subequations}
    \label{eq:reformulated_system}
   \begin{equation}
   \label{eq:Q_reformulated_formula}
       Q_t=M(L\Delta Q-rP(Q))
   \end{equation}
   \begin{equation}
   \label{eq:rt}
       r_t=P(Q):Q_t
   \end{equation}
\end{subequations}
subject to the same boundary and initial condition as \eqref{eq:Q_formula}.

In \cite{GWY2020}, we have constructed a fully discrete energy-stable scheme based on the IEQ formulation for this system and proved the convergence of the numerical solution to the weak solution of \eqref{eq:reformulated_system}. In this work, we will construct a semi-discrete numerical scheme solving the system \eqref{eq:reformulated_system}, following the idea we raised in \cite{GWY2020}. Let $(Q,r)$ denote the weak limit obtained by the convergence of numerical solutions. We will show a uniform $H^2$ bound for the Q-tensor given a regular enough initial condition and then deduce that $Q$ is a strong solution for system \eqref{eq:Q_formula}. 
It will be done by showing the equivalence of $r$ and $r(Q)$ in the $L^2$ sense. To the best of our knowledge, this is the first work to show such equivalence by explicitly computing the difference between the auxiliary variable and the energy quadratization in a discrete sense. The technique we use here can be applied to study other numerical schemes regarding different problems based on the IEQ method, for example, the hydrodynamical liquid system problem.


\section{Numerical Scheme}
Before starting the following analysis, we will introduce some important notations and definitions that appeared in this paper here. We denote the norm of Banach space $X$ as $\|\cdot\|_X$ and use $\|\cdot\|$ to denote the $L^2$ norm. The inner product on $L^2$ will be denoted as $\langle\cdot, \cdot\rangle$. We use $A:B$ to denote the Frobenius norm for two matrix-valued functions. If there is no specific illustration, we will assume a tensor product term to be trace-free and symmetric when we refer it to a Q-tensor. We list the lemmas frequently used in the following analysis: Lipschitz continuity of $P(Q)$\cite[Theorem 4.1]{GWY2020}, Agmon's inequality\cite[Chapter II, section 1.4]{temam1997infinite}, Sobolev's inequality\cite{evans10}, and Aubin-Lions lemma\cite{Simon_Aubin-Lions}. We say $Q$ is a strong solution for \eqref{eq:Q_formula} if $ Q\in L^2\left([0,T];H^3(\Omega)\bigcap H^1_0(\Omega)\right)\cap L^\infty\left([0,T];H^2(\Omega)\right)\bigcap H^1\left([0,T];H^1(\Omega)\right)$ and \eqref{eq:Q_formula} holds in $L^2([0,T]\times\Omega)$.

We propose the following numerical scheme to solve \eqref{eq:reformulated_system}:

\begin{empheq}[left = \empheqlbrace]{align}
               \frac{Q^{n+1}-Q^n}{\Delta t} &= MH^{n+1},\label{eq:DtQ} \\
               r^{n+1}-r^n &= P(Q^n):(Q^{n+1}-Q^n)\label{eq:Dtr},
\end{empheq}
where
$   H^{n+1}=L\Delta Q^{n+1}-r^{n+1}P(Q^n)$
subject to boundary condition $Q^n|_{\partial\Omega}=0$ or $\partial_n Q^n|_{\partial\dom}=0$

Firstly, we will state solvability of this numerical scheme. This can be summarized as the following lemma.

\begin{lemma}
    \label{lem:Solvability_scheme}
    The scheme \eqref{eq:DtQ}-\eqref{eq:Dtr} is well-defined. Specifically, given $Q^0\in H^2(\Omega)$ satisfying Dirichlet or Neumann boundary condition, for each $n>1$, we have $Q^n\in H^2(\Omega)$. Furthermore, for $n\geq 1$, we have $\rold\in H^1(\Omega)$.
\end{lemma}
\begin{proof}
Assume the lemma holds for $Q^n$, we will prove the lemma holds for $\Qnew$ as well. Substituting $r^{n+1}$ by \eqref{eq:Dtr} into $H^{n+1}$ and $\eqref{eq:DtQ}$, we obtain
\begin{equation}
    \label{eq:Q_update_formula}
    Q^{n+1}-ML\Delta Q^{n+1}\Delta t+(\Pold:\Qnew)\Pold\Delta t=Q^n-r^nP(Q^n)\Delta t+(\Pold:\Qold)\Pold\Delta t
\end{equation}
We define bilinear operator $T:H^1(\Omega)\times H^1(\Omega)\to \mathbb{R}$ as 
\begin{equation}
    \label{eq:update_operator}
    T(A, B)=\langle A, B\rangle+ML\langle \nabla A, \nabla B\rangle\, \Delta t + \langle \Pold, A\rangle\, \langle \Pold, B\rangle\,\Delta t.
\end{equation}
This operator is clearly bounded. For its ellipticity, taking $B=A$ and we yield
\begin{equation}
    \label{eq:ellipticity}
    \lvert T(A,A)\rvert =\|A\|^2 + ML\|\nabla A\|^2+\lvert \langle \Pold, A\rangle \rvert^2>C\|A\|_{H^1}^2,
\end{equation}
for some constant $C>0$ related to $M, L, \Delta t$. Here we have shown that $T$ is an elliptic operator. By standard regularity results for elliptic problems\cite{evans10, Grisvard}, there exists unique $\Qnew\in H^2(\Omega)$ solving \eqref{eq:DtQ}-\eqref{eq:Dtr}. The regularity of $r^{n}$ will immediately follow from the regularity of $Q^n$ and \eqref{eq:Dtr} by standard induction argument.
\end{proof}
Next, we will refer to the following results on trace-free and symmetric properties preserved by our numerical scheme and the discrete energy stability. Their proof follows in the same way as \cite[Proposition 4.4]{GWY2020} and \cite[Theorem 4.6]{GWY2020}, respectively.
\begin{lemma}
    \label{lem:sym_tracefree_property}
    If $\Qold$ is symmetric and trace-free, then $\Qnew$ computed by \eqref{eq:DtQ} is also symmetric and trace-free.
\end{lemma}
\begin{lemma}
    \label{lem:energy_stability}
    Define the discrete energy
    $ 
        E^n= \frac{L}{2}\|\nabla Q^n\|^2+\frac{1}{2}\|r^n\|^2$,
    then
    \begin{equation}
        \label{eq:discrete_energy_law}
        E^{n+1}-E^n\leq -M\|H^{n+1}\|^2\Delta t.
    \end{equation}
    In addition, it follows that for any $N>0$,
    $
        \sum_{n=0}^N\|H^{n+1}\|^2\,\Delta t\leq 2E_0$.
\end{lemma}

To the end of this section, we will show the following estimate for $Q^n$:
\begin{lemma}\label{lem:Delta_Q_estimate}
    If $Q^0\in H^2(\Omega)$, then $ 
        \sum_{k=1}^N \|\Delta Q^k\|^2\,\Delta t\leq M$,
    for some constant $M>0$
    \end{lemma}
    \begin{proof}
    Using definition of $H^{n+1}$, we have
    \begin{equation}
        \label{eq:Laplace_Q_L2}
        \begin{aligned}
            \|\Delta Q^{k+1}\|^2=\|\frac{1}{L}\left(H^{k+1}+r^{k+1}P(Q^k)\right)\|^2
            \leq C\left(\|H^{k+1}\|^2+\|P(Q^k)\|_{L^\infty}^2\|r^{k+1}\|^2 \right).
      \end{aligned}
    \end{equation}
It follows from Lemma \ref{lem:energy_stability}, Agmon's inequality and Lipschitz continuity of $P(Q)$ that
\begin{equation}
\label{eq:Laplace_Q_L2_2}
\begin{aligned}
 \|\Delta Q^{k+1}\|^2
 &\leq C\left(\|H^{k+1}\|^2+\|\Delta Q^k\|+1 \right)\leq C\left(1+\|H^{k+1}\|^2\right)+\frac{1}{2}\|\Delta Q^k\|^2.
\end{aligned}
\end{equation}
Multiplying $\Delta t$ on both sides and summing from $k=0$ to $k=N-1$, we have
    \begin{equation}
        \label{eq:Delta_Q_sum_estimate}
        \frac{1}{2}\|\Delta Q^N\|^2\Delta t+\frac{1}{2}\sum_{k=1}^N\|\Delta Q^k\|^2\Delta t\leq \frac{1}{2}\|\Delta Q^0\|^2\Delta t+\sum_{k=1}^N C\left(1+\|H^{k+1}\|^2\right)\Delta t,
    \end{equation}
    which is bounded uniformly in $\Delta t$ thanks to the discrete energy estimate given in Lemma \ref{lem:energy_stability}.
    
    \end{proof}

\section{Convergence Analysis}
\subsection{Higher order energy inequality}
Regarding the existence and regularity of results on Q-tensor models, one common condition necessary to obtain a strong solution is the uniform $H^2$ bound of the Q-tensor in time. However, the derivation of this result highly depends on the integrability of $r(Q)$. Specifically, $r(Q)$ is a quadratic function concerning $Q$. Its regularity follows from properties of $Q$. It is hard to be obtained in our reformulated system because it is not apparent to conclude that $r^k=O(|Q^k|^2)$. The only existing regularity we can use is the $L^2$ integrability of $r^k$. 

We will address this problem by explicitly estimating the difference between the auxiliary variable $r^k$ and $r(Q^k)$ for each $k\in\mathbb{N}$. By Taylor expansion for matrix-valued functions\cite{turnbull_1930}, 
\begin{equation}\label{eq:rQtaylor}
    r(Q^{k+1})-r(Q^k)=P(Q^k):(Q^{k+1}-Q^k)+R_k=r^{k+1}-r^k+R_k,
\end{equation}
where the remainder $R_k$ satisfies
$
    \lvert R_k\rvert
    \leq C\lvert Q^{k+1}-Q^k\rvert^2$
for some $C>0$ due to the Lipschitz continuity of $P$.
Then it follows from \eqref{eq:rQtaylor} that
\begin{equation}
    \label{eq:r_difference}
    \left(r^{k+1}-r(Q^{k+1})\right) -\left(r^k-r(Q^k)\right)=R_k.
\end{equation}
Taking sum from $k=0$ to $k=n-1$ and using the fact that $r^0=r(Q^0)$, it gives
\begin{equation}\label{eq:V_n}
    V_n\coloneqq\lvert r^n-r(Q^n)\rvert=\left\lvert\sum_{k=0}^{n-1} R_k\right\rvert\leq C\sum_{k=0}^{n-1}\left\lvert Q^{k+1}-Q^k\right\rvert^2.
\end{equation}
Lemma \ref{lem:Solvability_scheme} ensures us to take gradient of \eqref{eq:Dtr} to get
\begin{equation}
    \label{eq:gradient_r}
    \nabla r^{k+1}-\nabla r^k=\nabla \left(P(Q^k):(Q^{k+1}-Q^k)\right)=R_k^{(1)}+P(Q^k):(\nabla Q^{k+1}-\nabla Q^k),
\end{equation}
where 
$
    \lvert R_k^{(1)}\rvert= \lvert \nabla P(Q^k)\rvert\,\lvert Q^{k+1}-Q^k\rvert\leq C\lvert\nabla Q^k\rvert\,\lvert Q^{k+1}-Q^k\rvert.$ 
 Using chain rule, we have
\begin{equation}\label{eq:gradient_r(Q)}
\begin{aligned}
     &\quad\,\,\nabla r(Q^{k+1})-\nabla r(Q^k)\\
     &=P(Q^{k+1}):\nabla Q^{k+1}-P(Q^k):\nabla Q^k\\ &=\left(P(Q^{k+1})-P(Q^k)\right)\,\nabla Q^{k+1}+P(Q^k):(\nabla Q^{k+1}-\nabla Q^k)\coloneqq R_k^{(2)}+P(Q^k):(\nabla Q^{k+1}-Q^k).
\end{aligned}
\end{equation}
Using Lipschitz continuity of $P$, we obtain $
    \lvert R_k^{(2)}\rvert \leq C\lvert \nabla Q^{k+1}\rvert\lvert Q^{k+1}-Q^k\rvert$.
Now we yield
\begin{equation}
    \label{eq:nablarestimate}
    \left(\nabla r^{k+1}-\nabla r(Q^{k+1})\right)-\left(\nabla r^k-\nabla r(Q^k)\right)=R_k^{(1)}-R_k^{(2)}.
\end{equation}
Taking sum over $k$ from $k=0$ to $k=n-1$, we have
\begin{equation}
\label{eq:D_n}
    D_n\coloneqq\lvert\nabla r^n-\nabla r(Q^n)\rvert=\left\lvert \sum_{k=0}^{n-1} (R_k^{(1)}-R_k^{(2)})\right\rvert\leq C\sum_{k=0}^{n-1}\left[(\lvert \nabla Q^k\rvert + \lvert \nabla Q^{k+1}\rvert)\,\lvert Q^{k+1}-Q^k\rvert\right].
\end{equation}
To obtain strong solution of the system, we need higher order energy estimate for the discrete solution $Q^n$, more detailedly, we will try to deduce a uniform $H^2$ bound of $\Qnew$ which leads to
\begin{equation}
    \label{eq:Laplace_Q_estimate}
    \begin{aligned}
   \frac{L}{2}\|\Delta Q^{n+1}\|^2-\frac{L}{2}\|\Delta Q^n\|^2&\leq\langle L\Delta Q^{n+1}, \Delta Q^{n+1}-\Delta Q^n\rangle\\
    &=\langle H^{n+1}-r^{n+1}P(Q^n),\, \Delta Q^{n+1}-\Delta Q^n\rangle\\
    &=-\frac{1}{M}\,\|\frac{\nabla Q^{n+1}-\nabla Q^n}{\Delta t} \|^2\Delta t+\langle \nabla\left(r^{n+1}P(Q^n)\right),\,\frac{\nabla Q^{n+1}-\nabla Q^n}{\Delta t }\rangle\,\Delta t\\
    &\leq -\frac{1}{2M}\|\frac{\nabla Q^{n+1}-\nabla Q^n}{\Delta t}\|^2\Delta t+C\|\nabla\left(r^{n+1}P(Q^n)\right) \|^2\Delta t.
    \end{aligned}
\end{equation}
It intrigues us to provide the following estimate for $\|\nabla(\rnew\Pold)\|$.

\begin{lemma}
\label{lem:Estimate_of_gradient_rPQ}
    For fixed $\Delta t>0$ and each $n\in[0,\floor{\frac{T}{\Delta t}}]$, we have for some constant $C>0$ such that
    \begin{equation}
        \label{eq:Estimate_of_gradient_rPQ}
        \|\nabla\left(r^{n+1}P(Q^n)\right)\|^2\leq C+C(1+\|\Delta Q^n\|^2)(1+W_n+W_{n+1}+W_n^2\Delta t+W_{n+1}\Delta t)
    \end{equation}
\end{lemma}

\begin{proof}
 
We split this term by using  \eqref{eq:V_n} and \eqref{eq:D_n} to approximate, that is,
\begin{equation}
    \label{appen:nabla_rP}
   \begin{aligned}
   &\quad\,\, \|\nabla\left(r^{n+1}P(Q^n)\right)\|^2\\
   &=\|\nabla r^{n+1}P(Q^n)+(r^{n+1}-r^n)\,\nabla P(Q^{n})+r^n\nabla P(Q^{n})\|^2\\
   &=\|\nabla r^{n+1}P(Q^n)+P(Q^n):(\Qnew-\Qold)\,\nabla P(Q^{n})+r^n\nabla P(Q^{n})\|^2\\
    &\leq C\Big(\| \nabla r(Q^{n+1})\,P(Q^n)\|^2+\|r(Q^{n})\,\nabla P(Q^n)\|^2+\|P(Q^n):(\Qnew-\Qold)\,\nabla P(Q^{n})\|^2\\
    &\quad\quad\quad\,\,+\|D_{n+1}\,P(Q^n)\|^2+\|V_{n}\,\nabla P(Q^n)\|^2\Big)\coloneqq C \sum_{m=1}^5 I_m.
    \end{aligned}
\end{equation}
We will estimate $I_1$ to $I_5$ separately. For $I_1$, taking derivative of $r$ and use Lipschitz continuity of $P(Q)$, we have
\begin{equation}
    \label{eq:I1_1}
    \begin{aligned}
    I_1=\int_\Omega \lvert \nabla r(Q^{n+1})\rvert^2\,\lvert P(Q^n)\rvert^2\,dx&\leq\int_\Omega \lvert P(Q^{n+1})\rvert^2\,\lvert\nabla Q^{n+1}\rvert^2\,\lvert P(Q^n)\rvert^2\,dx\leq C\int_\Omega\lvert \nabla Q^{n+1}\rvert^2 \lvert Q^{n+1}\rvert^2\,\lvert Q^n\rvert^2\,dx.
    \end{aligned}
\end{equation}
Then by Holder's inequality and Sobolev's inequality, we obtain
\begin{equation}
    \label{eq:I1_2}
    \begin{aligned}
    I_1\leq C\|\nabla \Qnew\|_{L^6}^2\,\|\Qnew\|_{L^6}^2\,\|\Qold\|_{L^6}^2
    &\leq C\|Q^{n+1}\|_{H^2}^2\,\|Q^{n+1}\|_{H^1}^2\,\|Q^n\|_{H^1}^2\leq C(1+\|\Delta Q^{n+1}\|^2 )
    \end{aligned}
\end{equation}
We will treat $I_2$ in similar ways. Lipschitz continuity and Holder's inequality leads to
\begin{equation}
    \label{eq:I2_1}
    \begin{aligned}
    I_2=\int_\Omega \lvert r(Q^{n})\rvert^2\,\lvert \nabla P(Q^n)\rvert^2\,dx&\leq C\int_\Omega \lvert r(Q^{n})\rvert^2\,\lvert \nabla Q^n\rvert^2\,dx\leq C\left(\int_\Omega \lvert r(Q^{n})\rvert^3\,dx\right)^{\frac{2}{3}}\,\left(\int_\Omega \lvert \nabla Q^n\rvert^6\,dx\right)^{\frac{1}{3}}.
    \end{aligned}
\end{equation}
It then follows from definition of $r(\Qold)$ and Sobolev's embedding theorem that
\begin{equation}
    \label{eq:I2_2}
    I_2
    \leq C(1+\|Q^{n}\|^2_{L^6})\,\|Q^n\|_{H^2}^2 \leq C\|Q^{n}\|^2_{H^1}\,\|Q^n\|_{H^2}^2\leq C(1+\|\Delta Q^n\|^2).
\end{equation}
To bound the term $I_3$, we notice that
\begin{equation}
    \label{eq:I3_1}
    \begin{aligned}
    I_3=\int_\Omega \lvert \Pold:(\Qnew-\Qold)\rvert^2\,\lvert \nabla\Pold\rvert^2\,dx\leq C\int_\Omega \lvert \Qold\rvert^2\,\lvert \Qnew-\Qold\rvert^2\,\lvert \nabla \Qold\rvert^2\,dx,
    \end{aligned}
\end{equation}
and we can deduce
\begin{equation}
    \label{eq:I3_2}
    \begin{aligned}
    I_3\leq C \|\Qold\|_{L^6}^2\,\|\Qnew-\Qold\|_{L^6}^2\,\|\nabla \Qold\|^2_{L^6}&\leq C\|\Qold\|_{H^1}^2\|\Qold\|_{H^2}^2\left\|\frac{\nabla \Qnew-\nabla\Qold}{\Delta t}\right\|^2\,\Delta t^2\\
    &\leq C(1+\|\Delta \Qold\|^2)\,\Delta t\,\left[\sum_{k=0}^{n}\left\|\frac{\nabla Q^{k+1}-\nabla Q^k}{\Delta t}\right\|^2\,\Delta t\right]
    \end{aligned}
\end{equation}
The remaining problem is to control the error terms generated from introduction of the auxiliary variable. For term $I_4$, we have
\begin{equation}
    \label{eq:I4_1}
    \begin{aligned}
    I_4=\int_\Omega \lvert D_{n+1}\rvert^2\,\lvert P(Q^n)\rvert^2\,dx\leq C\int_\Omega\left[ \sum_{k=0}^n\left(\lvert \nabla Q^k\rvert+\lvert \nabla Q^{k+1}\rvert\right)\,\left\lvert Q^{k+1}-Q^k\right\rvert \right]^2\,\lvert Q^n\rvert^2\,dx.
    \end{aligned}
\end{equation}
Bounding $\lvert Q^n\rvert$ by $\|Q^n\|_{L^\infty}$ and using Cauchy-Schwarz inequality, we will obtain
\begin{equation}
    \label{eq:I4_2}
    I_4
    \leq C \|Q^n\|_{L^{\infty}}^2\sum_{k=0}^n\left[\left(\int_\Omega \lvert \nabla Q^k\rvert^4\,dx\right)^{\frac{1}{2}}+\left(\int_\Omega\lvert \nabla Q^{k+1}\rvert^4\,dx\right)^{\frac{1}{2}}\right]\left(\int_\Omega \lvert Q^{k+1}-Q^k\rvert^4\,dx\right)^{\frac{1}{2}}.
\end{equation}
Agmon's inequality and Sobolev's inequality enable us to estimate $I_4$ by
\begin{equation}
    \label{eq:I4_3}
    \begin{aligned}
    I_4&\leq C(1+\|\Delta Q^n\|)\,\sum_{k=0}^n\left( 1+\|\Delta Q^k\|^2+\|\Delta Q^{k+1}\|^2 \right)\,\left\|\frac{\nabla Q^{k+1}-\nabla Q^k}{\Delta t}\right\|^2\,\Delta t^2\\
    &\leq C(1+\|\Delta Q^n\|^2)
    \,\left[\sum_{k=0}^n\left\|\frac{\nabla Q^{k+1}-\nabla Q^k }{\Delta t}\right\|^2 \,\Delta t\right],
    \end{aligned}
\end{equation}
where we have used Lemma \ref{lem:Delta_Q_estimate} in the last inequality. Using Cauchy-Schwarz inequality, we yield
\begin{equation}
    \label{eq:I5}
    \begin{aligned}
    I_5=\int_\dom \lvert V_{n} \rvert^2\lvert\nabla P(Q^n)\rvert^2\,dx \leq C\int_\dom \left[ \sum_{k=0}^{n-1}\lvert Q^{k+1}-Q^k\rvert^2  \right]^2\,\lvert \nabla Q^n\rvert^2\,dx\leq \frac{C}{\Delta t}\sum_{k=0}^{n-1} \int_\dom \lvert \nabla Q^n\rvert^2\lvert Q^{k+1}-Q^k\rvert^4\,dx.
    \end{aligned}
\end{equation}
Applying Holder's inequality and Sobolev's inequality, we get
\begin{equation}
    \label{eq:I5_2}
    \begin{aligned}
    I_5&\leq \frac{C}{\Delta t}\sum_{k=0}^{n-1}\left(\int_\dom \lvert \nabla Q^n\rvert^6\,dx\right)^{\frac{1}{3}}\,\left( \int_\dom\lvert Q^{k+1}-Q^k\rvert^6\,dx \right)^{\frac{2}{3}}\\
    &\leq \frac{C}{\Delta t}\sum_{k=0}^{n-1}\|Q^n\|_{H^2}^2\left\|\frac{\nabla Q^{k+1}-\nabla Q^k}{\Delta t}\right \|^4\Delta t^4\leq C\Delta t(1+\|\Delta Q^n\|^2)\,\left(\sum_{k=0}^{n-1}\|\frac{\nabla Q^{k+1}-\nabla Q^k}{\Delta t}\|\Delta t\right)^2
    \end{aligned}
\end{equation}
Combining theses estimates together, we have shown the lemma.

\end{proof}

The inner products that appeared in the formula are well-defined due to regularity results given in Lemma \ref{lem:Solvability_scheme}.
Then we immediately have the following lemma, which provides the desired bound for the Q-tensor term.
\begin{lemma}
    \label{lem:Wbound}
    We define
    \begin{equation}
    \label{eq:Wn}
    W_n\coloneqq \frac{L}{2}\,\|\Delta Q^n\|^2+\frac{1}{2M}\sum_{k=0}^{n-1}\|\frac{\nabla Q^{k+1}-\nabla Q^k}{\Delta t}\|^2\,\Delta t,
\end{equation}
with $n\geq 1$. Then for small enough $\Delta t$, there exists $C=C(W_0,T)$ such that
$
    W_n\leq C(W_0,T)$
for any $n>1$.
\end{lemma}
\begin{proof}
We will prove this lemma by induction. Firstly, noting that \eqref{eq:Laplace_Q_estimate} is equivalent to
\begin{equation}
    \label{eq:Sn_estimate}
  \frac{  W_{n+1}-W_n}{\Delta t}\leq C\|\nabla\left(r^{n+1}P(Q^n)\right) \|^2.
\end{equation}
We infer from Lemma \ref{lem:Estimate_of_gradient_rPQ} that
\begin{equation}
\label{eq:W_increment}
    W_{n+1}-W_n\leq C\Delta t +C(1+\|\Delta \Qold\|^2)\,\Delta t\,(1+W_{n+1}+W_n+W_n^2\Delta t+W_{n+1}\Delta t).
\end{equation}
for every $n>0$. Assuming $W_n$ is bounded for all $n\leq N-1$, we will prove that $W_N$ is also bounded by the same constant. We choose $\Delta t\leq \max\limits_{n\leq N-1}\{1, \frac{1}{W_n}\}$. Later we will show that such $\Delta t$ exists uniformly for all $n$.  Summing \eqref{eq:W_increment} from $n=0$ to $N-1$, we get
\begin{equation}
    \label{eq:W_sum_estimate}
    W_N\leq W_0+CT+C\Delta t\,\sum_{n=0}^{N-1}(1+\|\Delta \Qold\|^2)\,(W_{n+1}+W_n)
\end{equation}
By applying discrete Gronwall's inequality\cite[Lemma 2.1]{Shen_projection_error1996} and using the energy estimate Lemma \ref{lem:Delta_Q_estimate}, we conclude that
\begin{equation}
    \label{eq:W_bound}
    W_n\leq (W_0+CT)\,e^{C(M, T)},
\end{equation}
where $M$ is the constant (upper bound) given in Lemma \ref{lem:Delta_Q_estimate}. This is a uniform bound for $W_n$ independent of $n$ which implies that $\frac{1}{W_n}$ is lower bounded. So appropriate $\Delta t$ can always be found. Here we have shown that $W_n\leq (W_0+CT)\,e^{C(M, T)}\coloneqq C(W_0, T)$ for all $n>0$.
\end{proof}

This is a uniform $H^2$ estimate for Q-tensor term and provides higher regularity result of the derivative in discrete sense. To the end of this part, we turn to investigate the relations between $\nabla H^{n+1}$ and $\nabla \Delta Q^{n+1}$. Firstly, for each fixed $\Delta t$, $\nabla H^{n+1} = \frac{\nabla \Qnew-\nabla \Qold}{M\Delta t}\in H^1(\Omega)$. It implies $L\nabla \Delta \Qnew = \nabla H^{n+1}+\nabla(\rnew\Pold)\in L^2$ is well-defined.
Then by Lemma \ref{lem:Wbound} and Lemma \ref{lem:Estimate_of_gradient_rPQ}, we can bound $\|\nabla(\rnew\Pold)\|$ by constant and so 
\begin{equation}
    \label{eq:Q_H3_estimate}
    \sum_{n=0}^N\|\nabla \Delta Q^{n+1}\|^2\,\Delta t\leq C\left( \|H^{n+1}\|^2 + \|\nabla(\rnew\Pold)\|^2\right)\,\Delta t\leq C(1+\|H^{n+1}\|^2)\,\Delta t.
\end{equation}
This can be summarized as the following lemma.
\begin{lemma}
\label{lem:Q_H3_estimate}
The numerical solutions obtained from scheme \eqref{eq:DtQ}-\eqref{eq:Dtr} satisfies,
$
    \sum_{n=0}^N\|\nabla \Delta Q^{n+1}\|^2\,\Delta t<C$,
for some constant $C>0$ and $N=\floor{\frac{T}{\Delta t}}$.
\end{lemma}

\subsection{Convergence to Strong Solution}

We construct linear interpolation numerical solution as 
\begin{equation}
    \label{eq:Qrsol}
    \Qsol(t) = \sum_{n=0}^{N-1}  \left[\alpha_{n+1}(t)\,\Qold+\alpha_n(t)\,\Qnew\right]\,\chi_{S_n},\quad\quad
    \rsol(t) = \sum_{n=0}^{N-1}  \left[\alpha_{n+1}(t)\,r^n+\alpha_n(t)\,r^{n+1} \right]\,\chi_{S_n},
\end{equation}
where $\alpha_n(t)=\frac{t-n\Delta t}{\Delta t}$,  $\alpha_{n+1}(t)=\frac{(n+1)\Delta t-t}{\Delta t}$, $S_n=[n\Delta t, (n+1)\Delta t)$ and $\chi_{S_n}$ is the characteristic function on $S_n$.

\begin{theorem}
    \label{thm:Strong_solution}
    Given initial value $Q_{in}\in H^2(\Omega)$, for fixed $T>0$, there exists a subsequence of numerical solution, denoted by $\{Q_{\Delta t_m}\}_m$ such that
$Q_{\Delta t_m}\to Q$ in $L^2(0,T;H^2(\Omega))$ and $Q$ is a strong solution for equation \eqref{eq:Q_formula}. 
\end{theorem}
\begin{proof}
    We infer from Lemma \ref{lem:Wbound} and \ref{lem:Q_H3_estimate} that
    \begin{equation}
        \label{eq:Q_sol_bound}
        \Qsol\in L^\infty(0,T;H^2(\Omega))\bigcap L^2(0,T; H^3(\Omega)),\quad\quad        \partial_t \Qsol\in L^\infty(0,T; L^2(\Omega))\bigcap L^2(0,T;H^1(\Omega)).
    \end{equation}
    The Sobolev's embedding implies that $\Qsol\in L^\infty([0,T]\times \dom)$ and Lipschitz continuity of $P(Q)$ leads to $P(\Qsol)\in L^\infty([0,T]\times \dom)$. Applying Aubin-Lions lemma and Banach–Alaoglu theorem, we conclude that there exists a pair of subsequence $\{(\Qsolsub, \rsolsub)\}_m$ and a pair of functions $(Q,r)$ such that
\begin{equation}
\label{eq:existence_strong_solution}
    Q\in L^2\left(0,T;H^3(\Omega)\right)\bigcap L^\infty\left(0,T;H^2(\Omega)\right)\bigcap H^1\left(0,T;H^1(\Omega)\right),\quad\quad r\in L^\infty([0,T]; L^2(\Omega)),
\end{equation}
and
\begin{equation}
    \label{eq:passing_the_limit}
  \begin{aligned}
    &Q_{\Delta t_m}\rightharpoonup Q \quad\text{in}\,\,L^2\left(0,T;H^3(\Omega)\right)\bigcap H^1(0,T;H^1(\Omega)),\quad\quad Q_{\Delta t_m}\overset{\ast} {\rightharpoonup} Q\quad\text{in}\,\,L^\infty(0,T;H^2(\Omega))\\
    & Q_{\Delta t_m}\to Q\quad\text{in}\,\,L^2(0,T;H^2(\Omega))\bigcap C(0,T;H^1(\Omega)),\quad\quad \rsolsub\overset{\ast} {\rightharpoonup} r\quad\text{in }  L^\infty([0,T];L^2(\Omega))
  \end{aligned}
\end{equation}

To see that $r=r(Q)$, we introduce $r(Q)_{\Delta t}=\sum_{n=0}^{N-1}  \left[\frac{(n+1)\Delta t-t}{\Delta t}\,r(Q^n)+\frac{t-n\Delta t}{\Delta t}\,r(Q^{n+1}) \right]\,\chi_{S_n}$. Noting that as long as $Q$ is uniformly bounded, $r(Q)$ will be Lipschitz continuous where the Lipschitz constant is bounded by $C(W_0, T)$ stated in Lemma \ref{lem:Wbound}. Then for every $t\in[0,T]$, assume $t\in[(n-1)\Delta t, n\Delta t)$ for some $n\geq1$, we have
\begin{equation}
\label{eq:r_limit}
\begin{aligned}
\|\rsolsub-r(Q)\|^2&\leq 2\|\rsolsub-r(Q)_{\Delta t_m}\|^2+2\|r(Q)_{\Delta t_m}-r(Q)\|^2\\
&\leq C\max\limits_{1\leq n\leq N}\left(\|r^{n-1}-r(Q^{n-1})\|^2+\|r^n-r(Q^n)\|^2\right)+C\|Q-\Qsolsub\|^2\\
&\leq C\left\|\sum_{k=0}^{N-1}\lvert Q^{k+1}-Q^k\rvert^2\right\|^2+C\|Q-\Qsolsub\|^2\\
&\leq C\Delta t_m\,\left(\sum_{k=0}^{N-1}\left\|\frac{\nabla Q^{k+1}-\nabla Q^k}{\Delta t_m}\right\|^2\,\Delta t_m\right)^2+C\|Q-\Qsolsub\|^2\to0,
\end{aligned}
\end{equation}
as $m\to \infty$ where we have used Lemma \ref{lem:Wbound}. So we have shown that $r=r(Q)$ and $\rsolsub\to r(Q)$ in $C([0,T];L^2(\Omega))$. By passing the limit to infinity, we can see that $Q_t=M(L\Delta Q-rP(Q))$ and with $rP(Q)=r(Q)P(Q)=S(Q)$, we obtain that $Q$ is a strong solution for equation \eqref{eq:Q_formula}.

\end{proof}

\bibliographystyle{abbrv}
\bibliography{relatedliterature}

\begin{thebibliography}{10}

\bibitem{AbelsLiu2013}
H.~Abels, G.~Dolzmann, and Y.~Liu.
\newblock Strong solutions for the beris-edwards model for nematic liquid
  crystals with homogeneous dirichlet boundary conditions.
\newblock {\em Advances in Differential Equations}, 21, 12 2013.

\bibitem{AbelsLiu2014}
H.~Abels, G.~Dolzmann, and Y.~Liu.
\newblock Well-posedness of a fully coupled navier--stokes/q-tensor system with
  inhomogeneous boundary data.
\newblock {\em SIAM Journal on Mathematical Analysis}, 46(4):3050--3077, 2014.

\bibitem{Ball2017}
J.~M. Ball.
\newblock Mathematics and liquid crystals.
\newblock {\em Molecular Crystals and Liquid Crystals}, 647(1):1--27, 2017.

\bibitem{CavaterraWu}
C.~Cavaterra, E.~Rocca, H.~Wu, and X.~Xu.
\newblock Global strong solutions of the full navier--stokes and \$q\$-tensor
  system for nematic liquid crystal flows in two dimensions.
\newblock {\em SIAM Journal on Mathematical Analysis}, 48(2):1368--1399, 2016.

\bibitem{collings1997introduction}
P.~Collings and M.~Hird.
\newblock {\em Introduction to Liquid Crystals: Chemistry and Physics}.
\newblock Liquid Crystals Book Series. CRC Press, 1997.

\bibitem{de1993physics}
P.~de~Gennes and J.~Prost.
\newblock {\em The Physics of Liquid Crystals}.
\newblock International Series of Monographs on Physics. Clarendon Press, 1993.

\bibitem{evans10}
L.~C. Evans.
\newblock {\em Partial differential equations}.
\newblock American Mathematical Society, Providence, R.I., 2010.

\bibitem{Grisvard}
P.~Grisvard.
\newblock {\em Elliptic Problems in Nonsmooth Domains}.
\newblock Society for Industrial and Applied Mathematics, 2011.

\bibitem{GWY2020}
V.~M. Gudibanda, F.~Weber, and Y.~Yue.
\newblock Convergence analysis of a fully discrete energy-stable numerical
  scheme for the q-tensor flow of liquid crystals.
\newblock {\em SIAM Journal on Numerical Analysis}, 60(4):2150--2181, 2022.

\bibitem{Gonzalez}
F.~Guill\'{e}n-Gonz\'{a}lez and M.~\'{A}ngeles Rodr\'{I}guez-Bellido.
\newblock Weak time regularity and uniqueness for a \$q\$-tensor model.
\newblock {\em SIAM Journal on Mathematical Analysis}, 46(5):3540--3567, 2014.

\bibitem{Iyer2015}
G.~Iyer, X.~Xu, and A.~Zarnescu.
\newblock Dynamic cubic instability in a 2d q-tensor model for liquid crystals.
\newblock {\em Mathematical Models and Methods in Applied Sciences},
  25:1477--1517, 06 2015.

\bibitem{LEE2022108161}
H.~G. Lee, J.~Shin, and J.-Y. Lee.
\newblock Energy quadratization runge–kutta scheme for the conservative
  allen–cahn equation with a nonlocal lagrange multiplier.
\newblock {\em Applied Mathematics Letters}, 132:108161, 2022.

\bibitem{LiuWalkington2000}
C.~Liu and N.~J. Walkington.
\newblock Approximation of liquid crystal flows.
\newblock {\em SIAM Journal on Numerical Analysis}, 37(3):725--741, 2000.

\bibitem{LIU2019206}
Z.~Liu and X.~Li.
\newblock Efficient modified techniques of invariant energy quadratization
  approach for gradient flows.
\newblock {\em Applied Mathematics Letters}, 98:206--214, 2019.

\bibitem{Majumdar2010LandauDeGennes}
A.~Majumdar and A.~Zarnescu.
\newblock Landau–de gennes theory of nematic liquid crystals: the
  oseen–frank limit and beyond.
\newblock {\em Archive for Rational Mechanics and Analysis}, 196:227–280, 04
  2010.

\bibitem{Mori_1999}
H.~Mori, E.~C. Gartland, J.~R. Kelly, and P.~J. Bos.
\newblock Multidimensional director modeling using the q tensor representation
  in a liquid crystal cell and its application to the cell with patterned
  electrodes.
\newblock {\em Japanese Journal of Applied Physics}, 38(Part 1, No.
  1A):135--146, jan 1999.

\bibitem{Intro_Qtensor}
N.~J. Mottram and C.~J.~P. Newton.
\newblock Introduction to q-tensor theory, 2014.

\bibitem{Nochetto2017}
R.~H. Nochetto, S.~W. Walker, and W.~Zhang.
\newblock A finite element method for nematic liquid crystals with variable
  degree of orientation.
\newblock {\em SIAM Journal on Numerical Analysis}, 55(3):1357--1386, 2017.

\bibitem{PaicuZarnescuGlobalExistence}
M.~Paicu and A.~Zarnescu.
\newblock Global existence and regularity for the full coupled navier–stokes
  and q-tensor system.
\newblock {\em SIAM Journal on Mathematical Analysis}, 43(5):2009--2049, 2011.

\bibitem{PhysRevLett.59.2582}
N.~Schopohl and T.~J. Sluckin.
\newblock Defect core structure in nematic liquid crystals.
\newblock {\em Phys. Rev. Lett.}, 59:2582--2584, Nov 1987.

\bibitem{Shen_projection_error1996}
J.~Shen.
\newblock On error estimates of the projection methods for the navier-stokes
  equations: Second-order schemes.
\newblock {\em Math. Comput.}, 65:1039--1065, 07 1996.

\bibitem{ShenJie2019}
J.~Shen, J.~Xu, and J.~Yang.
\newblock A new class of efficient and robust energy stable schemes for
  gradient flows.
\newblock {\em SIAM Review}, 61(3):474--506, 2019.

\bibitem{Shin2008}
W.-J. Shin, S.-Y. Cho, J.-B. Lee, S.-H. Yoon, and T.-Y. Won.
\newblock Implementation of q-tensor model in three-dimensional finite element
  method simulator.
\newblock {\em Japanese Journal of Applied Physics}, 47:5561--5566, 07 2008.

\bibitem{Simon_Aubin-Lions}
J.~Simon.
\newblock Compact sets in the space lp(o,t; b).
\newblock {\em Annali di Matematica Pura ed Applicata}, 146:65--96, 01 1986.

\bibitem{temam1997infinite}
R.~Temam.
\newblock {\em Infinite-Dimensional Dynamical Systems in Mechanics and
  Physics}.
\newblock Applied Mathematical Sciences. Springer New York, 1997.

\bibitem{turnbull_1930}
H.~W. Turnbull.
\newblock A matrix form of taylor's theorem.
\newblock {\em Proceedings of the Edinburgh Mathematical Society},
  2(1):33–54, 1930.

\bibitem{WangWangZhang2017}
M.~Wang, W.~Wendong, and Z.~Zhang.
\newblock From the q-tensor flow for the liquid crystal to the harmonic map
  flow.
\newblock {\em Archive for Rational Mechanics and Analysis}, 225, 08 2017.

\bibitem{YANG2016294}
X.~Yang.
\newblock Linear, first and second-order, unconditionally energy stable
  numerical schemes for the phase field model of homopolymer blends.
\newblock {\em Journal of Computational Physics}, 327:294--316, 2016.

\bibitem{ZhaoYang2017}
X.~Yang and J.~Zhao.
\newblock On linear and unconditionally energy stable algorithms for variable
  mobility cahn-hilliard type equation with logarithmic flory-huggins
  potential.
\newblock 01 2017.

\bibitem{YangZhaoWangShen2017}
X.~Yang, J.~Zhao, Q.~Wang, and J.~Shen.
\newblock Numerical approximations for a three-component cahn–hilliard
  phase-field model based on the invariant energy quadratization method.
\newblock {\em Mathematical Models and Methods in Applied Sciences},
  27(11):1993--2030, 2017.

\bibitem{ZHAO2021107331}
J.~Zhao.
\newblock A revisit of the energy quadratization method with a relaxation
  technique.
\newblock {\em Applied Mathematics Letters}, 120:107331, 2021.

\bibitem{ZHAO2017803}
J.~Zhao, X.~Yang, Y.~Gong, and Q.~Wang.
\newblock A novel linear second order unconditionally energy stable scheme for
  a hydrodynamic q-tensor model of liquid crystals.
\newblock {\em Computer Methods in Applied Mechanics and Engineering},
  318:803--825, 2017.

\bibitem{ZhaoYangLiWang2016}
J.~Zhao, X.~Yang, J.~Li, and Q.~Wang.
\newblock Energy stable numerical schemes for a hydrodynamic model of nematic
  liquid crystals.
\newblock {\em SIAM Journal on Scientific Computing}, 38(5):A3264--A3290, 2016.

\end{thebibliography}

\end{document}